\documentclass[a4paper,UKenglish,cleveref, autoref, thm-restate]{lipics-v2021}
\pdfoutput=1

\usepackage{float} 
\usepackage{mathtools} 
\usepackage{ytableau}  
\usepackage{tikz}     
\usetikzlibrary{positioning, arrows}

\usepackage{cite}

\usepackage{todonotes}

\newcommand{\TC}{\mathrm{TC}}

\newcommand{\Bc}{\mathcal{B}}
\newcommand{\Cc}{\mathcal{C}}

\newcommand{\oeis}[1]{\href{http://oeis.org/#1}{#1}}

\newcommand{\E}{\mathbb{E}}

\title{A Combinatorial Framework for the Pons--Batle Identity: Young Tableaux, Lattice Paths, and Limit Laws}

\titlerunning{The Pons--Batle Identity: Young Tableaux, Lattice Paths, and Limit Laws}

\author{Hexuan Liu}{Department of Pure and Applied Mathematics, Waseda University, Japan  \and \url{https://researchmap.jp/HexuanLiu?lang=en} }{lhx@ruri.waseda.jp}{https://orcid.org/0009-0001-0388-6041}{Supported by JSPS KAKENHI Grant Number JP25KJ2168 and JST SPRING Grant Number JPMJSP2128}

\author{Michael Wallner}{Institute of Discrete Mathematics, TU Graz, Austria \and Institute of Discrete Mathematics and Geometry, TU Wien, Austria \and \url{https://dmg.tuwien.ac.at/mwallner/} }{michael.wallner@tuwien.ac.at}{https://orcid.org/0000-0001-8581-449X}{Partially supported by the Austrian Science Fund (FWF) P~34142 and AST1535024.}

\author{Guan-Ru Yu}{Department of Applied Mathematics, National Sun Yat-sen University, Taiwan  \and \url{https://sites.google.com/site/guanruyu0127/} }{gryu@math.nsysu.edu.tw}{https://orcid.org/0000-0003-4255-6974}{NSTC-113-2115-M-110-004-MY3 (National Science and Technology Council, Taiwan)}

\authorrunning{H.\ Liu, M.\ Wallner, and G.-R.\ Yu}

\Copyright{Hexuan Liu, Michael Wallner, and Guan-Ru Yu}



\category{Extended Abstract}

\relatedversion{}

\acknowledgements{We thank Mireille Bousquet-Mélou and Michael Fuchs for stimulating discussions regarding this work.}

\nolinenumbers

\EventEditors{Konstantinos Panagiotou}
\EventNoEds{1}
\EventLongTitle{37th International Conference on Probabilistic, Combinatorial and Asymptotic Methods for the Analysis of Algorithms (AofA 2026)}
\EventShortTitle{AofA 2026}
\EventAcronym{AofA}
\EventYear{2026}
\EventDate{June 22--26, 2026}
\EventLocation{Munich, Germany}
\EventLogo{}
\SeriesVolume{381}
\ArticleNo{26}

\begin{document}

\maketitle

\begin{abstract}
Tree-child networks are an important class of phylogenetic network used to model reticulate evolutionary processes. These networks have attracted increasing attention from researchers with interests in both combinatorics and algorithms. A fundamental open problem posed by Pons and Batle asks whether the number $\TC_{n,k}$ of bicombining tree-child networks with $n$ leaves and $k$ reticulation nodes equals the number of certain constrained words, now called Pons--Batle words. In this paper, we confirm the conjecture for tree-child networks with a bounded number of reticulation nodes.

Our approach is combinatorial and analytic. We introduce families of Young tableaux with walls and holes and construct explicit bijections with Pons--Batle words, yielding a direct combinatorial explanation of the identities. These tableaux encode structural features of the underlying networks, including the placement of reticulation nodes. By projecting them to decorated Dyck paths, we obtain algebraic generating functions with differential operators encoding step weights, leading to explicit recurrence relations and closed-form formulas for $\TC_{n,k}$.

Beyond finite verification for moderate $k$, the framework reveals an underlying probabilistic structure. For $k=1$, natural structural parameters, such as the position and value of distinguished cells, converge, after rescaling, to $\mathrm{Beta}(2,1)$, $\mathrm{Beta}(1,2)$, and Uniform (i.e., $\mathrm{Beta}(1,1)$) distributions. These limit laws arise from a coalescence of singularities at the dominant square-root singularity, producing a non-analytic transition in the local expansion.

Overall, our results provide both combinatorial insight and a unified analytic perspective on the asymptotic behavior of tree-child networks, showing how algebraic generating functions with interacting singularities systematically produce Beta limit laws.

\end{abstract}

\clearpage

\section{Reformulating the Pons--Batle Identity in terms of Young tableaux}

\subsection{History and Motivation}
In 2012, Alois P. Heinz introduced the sequence \oeis{A213863} in the OEIS\footnote{The On-Line Encyclopedia of Integer Sequences (OEIS): \url{https://oeis.org/}}, arising from the following combinatorial construction.

\begin{definition}\label{an}
Let $\mathcal{A}_n$ denote the class of words over the alphabet $\{\omega_1, \dots, \omega_n\}$ such that each letter occurs exactly three times and, in every prefix of a word, either $\omega_i$ has not yet appeared or its number of occurrences is at least as large as that of $\omega_j$ for all $j>i$. Let $a_n = |\mathcal{A}_n|$.
\end{definition}

For example,
$
\mathcal{A}_2 = \{aaabbb, aababb, abaabb, baaabb, aabbab, ababab, baabab\},
$
and the first few values are
$
\{a_n\}_{n\ge1} = \{1, 7, 106, 2575, 87595, 3864040, 210455470, \dots\}.
$

In the second half of 2020, Fuchs, Yu, and Zhang~\cite{fuchs2021asymptotic} discovered that $a_{n-1}$ also counts tree-child networks with $n$ leaves and $n-1$ reticulation nodes, once leaf labels are ignored. Here, a \emph{tree-child network} is a rooted directed acyclic graph whose internal nodes are either tree nodes (indegree~$1$, outdegree~$2$) or reticulation nodes (indegree~$2$, outdegree~$1$), such that every non-leaf node has at least one child that is not a reticulation node. More precisely, they established the following result.

\begin{proposition}
There is a bijection between the set $\mathcal{TC}_{n,n-1}$ of tree-child networks with $n$ leaves and $n-1$ reticulation nodes (with labels removed) and the class $\mathcal{A}_{n-1}$. Consequently,
\[
a_{n-1} = \frac{TC_{n,n-1}}{n!},
\]
where $TC_{n,n-1} = |\mathcal{TC}_{n,n-1}|$.
\end{proposition}

In mid-2021, Pons and Batle~\cite{BatlePons2021Phylo} introduced a more general family of words.

\begin{definition}[Pons--Batle words]\label{def:PonsBatleWords}
Let $\mathcal{C}_{n,k}$ denote the class of words over the alphabet $\{\omega_1, \dots, \omega_n\}$ such that the letters $\omega_1,\dots,\omega_k$ occur exactly three times, while the letters $\omega_{k+1},\dots,\omega_n$ occur exactly twice. Moreover, in every prefix of a word, either $\omega_i$ has not yet appeared or its number of occurrences is at least as large as that of $\omega_j$ for all $j>i$. Let $c_{n,k} = |\mathcal{C}_{n,k}|$.
\end{definition}

For example, we have
$
\mathcal{C}_{2,1} = \{aabba, ababa, baaba, aaabb, aabab, abaab, baaab\}.
$
By definition, $\mathcal{C}_{n,n} = \mathcal{A}_n$, and hence $a_n = c_{n,n}$. Pons and Batle formulated a striking conjecture relating these words to phylogenetic networks, which was recently proven by Lin et al.~\cite{linetall2026ProofPonsBatle}.

\begin{theorem}[Pons--Batle Identity, \cite{BatlePons2021Phylo,linetall2026ProofPonsBatle}]\label{thm:Pons-Batle}
The number $TC_{n,k}$ of tree-child networks with $n$ leaves and $k$ reticulation nodes satisfies
\begin{align}\label{eq:Pons-Batle}
TC_{n,k} = \frac{n!}{(n-k)!} \, c_{n-1,k}.
\end{align}
\end{theorem}

This identity is particularly remarkable because, prior to its discovery, no effective method was known for enumerating tree-child networks with prescribed numbers of leaves and reticulation nodes. In contrast, the quantities $c_{n,k}$ are easy to compute, as Pons and Batle~\cite{BatlePons2021Phylo} showed they satisfy the simple recurrence
\begin{align}
    \label{eq:reccnk}
    c_{n,k} = c_{n,k-1} + (2n+k-1)c_{n-1,k}.
\end{align}

Despite its seemingly simple form, the problem of proving \eqref{eq:Pons-Batle} remained open until early 2026, when Lin et al.~\cite{linetall2026ProofPonsBatle} established the identity in full generality. However, their proof is computational and non-bijective. As demonstrated by the work of Caraceni et al.~\cite{CARACENI2022112944}, bijective and combinatorial methods offer much deeper structural insights into tree-child networks. Thus, developing a combinatorial understanding of the Pons--Batle Identity remains an important open problem.

Significant effort has already been devoted to this pursuit. Shortly after the conjecture was proposed, Fuchs, Liu, and Yu~\cite{FuchsLiuYu2021Note} obtained partial bijective results for one-component networks:

\begin{proposition}[{\cite[Theorem~1]{FuchsLiuYu2021Note}}]\label{potc}
There exists a bijection
\[
\mathcal{OTC}_{n,k} \;\longleftrightarrow\; \mathcal{H}_{n-1,k} \times \{1, \dots, n\}^{\underline{k}},
\]
where $A^{\underline{k}}$ denotes the set of $k$-tuples of distinct elements of $A$. Consequently,
\[
|\mathcal{OTC}_{n,k}| = \frac{n!}{(n-k)!} \, |\mathcal{H}_{n-1,k}|,
\]
where $\mathcal{OTC}_{n,k}$ is the class of one-component tree-child networks with $n$ leaves and $k$ reticulation nodes, and $\mathcal{H}_{n,k}$ is the subset of words in $\mathcal{C}_{n,k}$ such that, before the third occurrence of any letter, all letters appear exactly twice.
\end{proposition}

Later, in 2024, Fuchs et al.~\cite{Changetal2023dcombining} proposed a different generalization of these word classes:

\begin{definition}\label{bnkd}
Let $\mathcal{B}_{n,k}$ denote the class of words over the alphabet $\{\omega_1, \dots, \omega_n\}$ in which $k$ letters occur three times and $n-k$ letters occur twice, subject to the following condition: in every prefix of a word, either $\omega_i$ has not yet appeared or its number of occurrences is at least as large as that of $\omega_j$ for all $j>i$. For letters that appear only twice, the first and second occurrences are treated as the second and third occurrences, respectively.
\end{definition}

As a minimal nontrivial example, $\mathcal{B}_{2,1} = \{aabbb, aabab, aaabb, babab, baabb, abbab, ababb\}$.
By definition, $\mathcal{B}_{n,n} = \mathcal{A}_n$, yielding $a_n = b_{n,n}$. Crucially, they proved the following enumeration result:

\begin{proposition}[{\cite[Theorem~3.4]{Changetal2023dcombining}}]\label{bnk}
Let $b_{n,k} = |\mathcal{B}_{n,k}|$. Then
\[
\mathrm{TC}_{n,k} = \frac{n!}{2^{n-k-1}} b_{n-1,k}.
\]
\end{proposition}

Although computing $b_{n,k}$ is more involved than computing $c_{n,k}$, Proposition~\ref{bnk} effectively completed the exact enumeration of tree-child networks. Nevertheless, the Pons--Batle Identity~\eqref{eq:Pons-Batle} remains highly appealing due to the simpler structure of $\mathcal{C}_{n,k}$. Combining Theorem~\ref{thm:Pons-Batle} and Proposition~\ref{bnk}, the core combinatorial challenge reduces to explaining the identity
\begin{align}\label{bceq}
    \frac{b_{n,k}}{2^{n-k}} = \frac{c_{n,k}}{(n+1-k)!}.
\end{align}

In this extended abstract, we study the combinatorial nature of Equation~\eqref{bceq}. We introduce a lattice path framework that not only provides an algorithmic proof of the identity for finite $k$ (which we verify for $k \leq 250$) but, more importantly, reveals the probabilistic behavior of these networks, yielding new limit laws. Although we do not completely resolve the problem bijectively, our methods open a robust pathway toward extensions to $d$-combining networks, to which our approach can be readily generalized.

\subsection{Young tableaux with walls}

In 2021, Banderier and Wallner~\cite{banderier2021young} established a bijection between these classes of constrained words and \emph{Young tableaux with walls}. In their framework, Young tableaux are generalized by allowing certain adjacent pairs of boxes to violate the usual monotonicity conditions; such violations are encoded by inserting a \emph{wall} between the corresponding boxes. Each tableau is filled with the integers $1,\dots,n$, where $n$ is the total number of boxes, so that entries increase from left to right along rows and from bottom to top along columns, except across walls, where no monotonicity condition is imposed.

Under this correspondence, each element of the class $\mathcal{A}_n$ in Definition~\ref{an} is mapped to a $3\times n$ Young tableau with walls (three rows and $n$ columns), in which walls are inserted between every pair of adjacent boxes in the bottom row (that is, the $n$ boxes are separated by $n-1$ walls). We denote this class of Young tableaux by $\mathcal{A}^*_n$; see Figure~\ref{a2bijection} in Appendix~\ref{apb} for an example.

It is worth emphasizing that not only does the class $\mathcal{A}_n$ admit such a correspondence, but all of the word classes introduced above can be treated in a similar way. For example, each element of the class $\mathcal{C}_{n,k}$ in Definition~\ref{def:PonsBatleWords} corresponds to a Young tableau with walls consisting of three rows, where the top row contains $k$ boxes and the other two rows contain $n$ boxes each, with all rows left-justified. As before, walls are inserted between every pair of adjacent boxes in the bottom row. We then consider all fillings of these tableaux with the integers $1,\dots,2n+k$. We denote this class of Young tableaux by $\mathcal{C}^*_{n,k}$; see Figure~\ref{c21bijection} in Appendix~\ref{apb} for an example. 

The class $\mathcal{B}_{n,k}$ in Definition~\ref{bnk} admits a similar correspondence. Here, the associated Young tableaux have three rows, where the top two rows each contain $n$ boxes, while the bottom row consists of $k$ boxes placed in arbitrary positions among the $n$ columns. Whenever two boxes in the bottom row are adjacent, a wall is inserted between them. As before, we consider all fillings of these tableaux with the integers $1,\dots,2n+k$. We denote this class of Young tableaux by $\mathcal{B}^*_{n,k}$; see Figure~\ref{b21bijection} in Appendix~\ref{apb} for an example.

Moreover, Figure~\ref{abcbijection} in Appendix~\ref{apb} shows more complex examples of the three classes in terms of Young tableaux.

One of our key steps is to embed both classes $\mathcal{B}^*_{n,k}$ and $\mathcal{C}^*_{n,k}$ into a $3$-parameter class $\mathcal{D}^*_{n,k,\ell}$.
Our goal is to construct a model that incorporates the features of both $\mathcal{B}^*_{n,k}$ and $\mathcal{C}^*_{n,k}$, so that each arises as a special case of $\mathcal{D}^*$. 

We define the class $\mathcal{D}^*_{n,k,\ell}$ ($k,\ell \le n$) as follows. Starting from a $3\times n$ Young tableau, as in $\mathcal{A}^*_n$, we remove the rightmost $n-k$ boxes from the top row and arbitrarily remove $n-\ell$ boxes from the bottom row, so that the resulting tableau contains $n+k+\ell$ boxes in total. As before, whenever two boxes in the bottom row are adjacent, a wall is inserted between them. Finally, we consider all fillings of such tableaux with the integers $1,\dots,n+k+\ell$ subject to the usual Young tableau conditions. 
For illustration, Figure~\ref{nonouter case2b} shows two elements of $\mathcal{D}^*_{6,3,2}$.

Clearly, under this definition we have $\mathcal{D}^*_{n,n,n}=\mathcal{A}^*_n$, $\mathcal{D}^*_{n,n,k}=\mathcal{B}^*_{n,k}$, and $\mathcal{D}^*_{n,k,n}=\mathcal{C}^*_{n,k}$. We write $d_{n,k,\ell}=|\mathcal{D}^*_{n,k,\ell}|$, and our goal is therefore to prove the identity
\begin{align*}
\frac{d_{n,n,k}}{2^{n-k}} = \frac{d_{n,k,n}}{(n+1-k)!}.
\end{align*}

\section{From Young tableaux to lattice paths}

\subsection{The Identity in terms of Young tableaux with walls and holes}

We now study the class $\mathcal{D}^*_{n,k,\ell}$ and its cardinality $d_{n,k,\ell}$ in more detail. For convenience, we reinterpret this model from a slightly different perspective.
 
Let $y_{k,\ell_1,\ell_2}$ denote the number of Young tableaux of the following shape (see Figure~\ref{nonouter case2b} for two examples): the top row consists of $k$ boxes in columns $1$ through $k$; the middle row consists of $n := \ell_1 + \ell_2$ boxes in columns $1$ through $n$; and the bottom row consists of $\ell_2$ boxes placed arbitrarily among columns $1$ through $n$, with the condition that adjacent boxes are separated by walls. 
Throughout, such a Young tableau is a standard filling with the integers $1,2,\dots,k+\ell_1+2\ell_2$ that increases along rows and columns, except that the left-to-right condition is waived for boxes in the bottom row separated by walls.

Under this reinterpretation, there is a simple correspondence 
$y_{k,\ell_1,\ell_2} = d_{\ell_1+\ell_2,k,\ell_2}.$

\begin{figure}[H]
    \centering
    \includegraphics[scale=0.65]{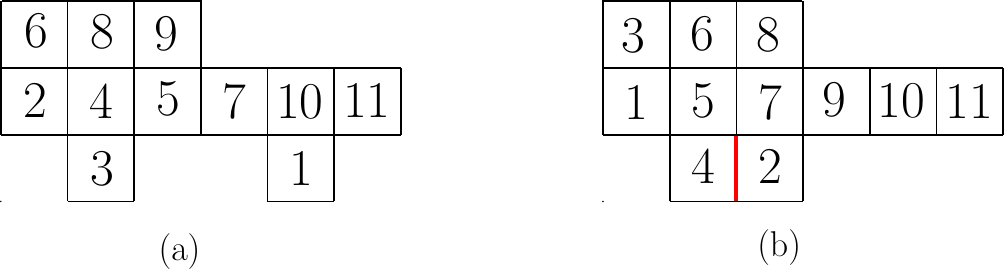}
    \caption{Two examples of Young tableaux with walls and holes of type $(k,\ell_1,\ell_2)=(3,4,2)$. Boxes in the third row are separated by walls. Both examples belong to $\mathcal{D}^*_{6,3,2}$.}
    \label{nonouter case2b}
\end{figure}

Pons--Batle words defined in Definition~\ref{def:PonsBatleWords} are in bijection with Young tableaux with walls and holes.

\begin{lemma}\label{lem:PBwords-legs}
For $n>k\ge 0$, the set $\mathcal{C}_{n-1,k}$ is in bijection with Young tableaux with walls and holes counted by $y_{k,0,n-1}$.
In particular, $|\mathcal{C}_{n-1,k}|=y_{k,0,n-1}$.
\end{lemma}

\begin{proof}
We describe a bijection between $\mathcal{C}_{n-1,k}$ and the set of Young tableaux counted by $y_{k,0,n-1}$.
Let $L := 2n+k-2$ be the length of the words.
Given a word $w \in \mathcal{C}_{n-1,k}$, we construct a tableau by placing the integers $1,2,\dots,L$ sequentially into the shape.
Reading $w$ from left to right, if the $m$-th letter is $\omega_j$, we place the integer $m$ into the next available box of column $j$ (filling from bottom to top).

Condition~(i) in Definition~\ref{def:PonsBatleWords} ensures that every column receives the correct number of entries, matching the shape of $y_{k,0,n-1}$.
By construction, the columns are strictly increasing.
Moreover, Condition~(ii) implies that for any $i<j$, the $r$-th occurrence of $\omega_i$ appears before the $r$-th occurrence of $\omega_j$.
This guarantees that the top two rows are strictly increasing; the third row imposes no horizontal condition due to the walls.

Conversely, given such a tableau, we recover the word by reading the entries $1,2,\dots,L$ in increasing order and recording the column index of each entry. These maps are inverses of each other.
\end{proof}

Therefore, the Pons--Batle Identity~\eqref{eq:Pons-Batle} is equivalent to 
\begin{align*}
	\TC_{n,k} = \frac{n!}{(n-k)!} y_{k;0,n-1}, \qquad  \text{ for } n > k \geq 0.
\end{align*}

By adapting the ideas from the proof of Lemma~\ref{lem:PBwords-legs} to the word classes in Proposition~\ref{potc} and Definitions~\ref{bnkd}, one can obtain analogous results for the subclass of one-component networks~\cite{FuchsLiuYu2021Note}. Moreover, in~\cite[Theorem~3.4]{Changetal2023dcombining} we proved bijectively that
\[
\TC_{n,k} = \frac{n!}{2^{n-k-1}}\, y_{n-1;n-k-1,k}, \qquad \text{for } n>k\ge 0.
\]

The main advantage of these reformulations is that the identity can be phrased entirely in terms of Young tableaux with walls.
Combining the last two mentioned results, we get\footnote{For simpler presentation we shifted the index $n$ by $1$.}
\begin{proposition}
The Pons--Batle Identity~\eqref{eq:Pons-Batle} is equivalent to
\begin{align}
    \label{eq:ConjPB_YT}
    y_{n,n-k,k} = \frac{2^{n-k}}{(n-k+1)!} y_{k,0,n}, \qquad  \text{ for } n \geq k \geq 0.
\end{align}
\end{proposition}

The advantage of the reformulation in terms of Young tableaux with walls and holes is that $y_{k,\ell_1,\ell_2}$ admits the following simple recurrence relation generalizing the recurrence~\eqref{eq:reccnk} of $c_{n,k}$.
It allowed us, for the first time, to verify the conjecture for all $0 \le n \le 300$.

\begin{lemma}
    The numbers $y_{k, \ell_1, \ell_2}$ of Young tableaux with walls and holes satisfy
    {
    \newcommand{\nrwop}[1]{\hspace{-0.075em} #1 \hspace{-0.075em}}
    \begin{align}
        \label{eq:rec_ykl1l2}
        y_{k, \ell_1, \ell_2} = y_{k-1,\ell_1,\ell_2} + y_{k, \ell_1-1, \ell_2} + (2\ell_2 \nrwop{+} \ell_1 \nrwop{+} k \nrwop{-} 1) y_{k, \ell_1, \ell_2-1},
        \qquad \text{for } \ell_1 \nrwop{+} \ell_2 \geq k > 0,
    \end{align}
    }
    with the initial condition $y_{0,0,0} = 1$ and the boundary condition $y_{k,\ell_1,\ell_2} = 0$ outside of the cone $\ell_1 + \ell_2 \geq k > 0$.
\end{lemma}

\begin{proof}
We decompose a Young tableau counted by $y_{k,\ell_1,\ell_2}$ according to the position of its maximal entry. 
The maximal entry must lie either in the top row or in the middle row.

First suppose that it lies in the top row. Removing this entry decreases the length of the top row by one. 
The resulting shape remains admissible only if the middle row is at least as long as the top row, that is, if $\ell_1 + \ell_2 \ge k$.

Now suppose that the maximal entry lies in the middle row. In this case there are two possible ways to obtain the tableau by extension: either we add a single box to the middle row, or we add a box to the middle row together with an additional box in the bottom row.

Because of the presence of the left and right walls, the entry placed in the bottom box is independent of the remaining entries. Hence it can be chosen arbitrarily among the $2\ell_2 + \ell_1 + k - 1$ values smaller than the maximum. After this choice, we relabel the entries in increasing order so as to obtain a standard Young tableau.

This yields the claimed recurrence relation.
\end{proof}

\subsection{The Pons--Batle Identity in terms of lattice paths}\label{sec:conjinlps}

The recurrence~\eqref{eq:rec_ykl1l2} can be interpreted in terms of lattice paths in three dimensions:
The numbers $y_{k,\ell_1,\ell_2}$ count the weighted lattice paths that start at the origin $(0,0,0)$ and end at $(k,\ell_1,\ell_2)$. 
Each path consists of unit steps $K=(1,0,0)$, $L_1=(0,1,0)$, and $L_2=(0,0,1)$, while never leaving the cone $\ell_1 + \ell_2 \ge k$.

The weight of a path is the product of the weights of its steps, where the steps $K$ and $L_1$ have weight $1$ and the step $L_2$ ending at $(k,\ell_1,\ell_2)$ has weight $2\ell_2 + \ell_1 + k-1$.

Now consider the reformulation~\eqref{eq:ConjPB_YT} of the identity and in particular the two appearing families of Young tableaux enumerated by $y_{n,n-k,k}$ and $y_{k,0,n}$.
Both shapes consist of $2n+k$ boxes, but in terms of lattice paths their lengths differ.
The first one ending at $(n,n-k,k)$ takes $2n$ steps and the second one ending at $(k,0,n)$ takes $n+k$ steps.
Also note that the sum of $L_1$ and $L_2$ steps, which is important for the cone condition, is equal to $n$ and independent of $k$. 

We now reduce the problem to two dimensions by grouping the steps $L_1$ and $L_2$. 
For this purpose we project the step $K$ onto $I = (1,0)$, 
the step $L_1$ onto $J_1=(0,1)$, 
and the step $L_2$ onto $J_2=(0,1)$. 
The steps $I$ and $J_1$ have weight $1$, 
while the weight of $J_2$ depends on the current trajectory. 
Specifically, the weight $w_{i,j,k}$ of the $k$-th $J_2$ step ending at $(i,j)$ is 
\begin{align}
    \label{eq:wijk}
    w_{i,j,k} = i+j+k-1.
\end{align}
Furthermore, the paths start at the origin $(0,0)$ and all its points $(i,j)$ satisfy $i \leq j$, i.e., they always stay above or on the diagonal $y=x$; see Figure~\ref{fig:projectedpath}.

\begin{figure}
  \centering
  \includegraphics[width=.4\linewidth]{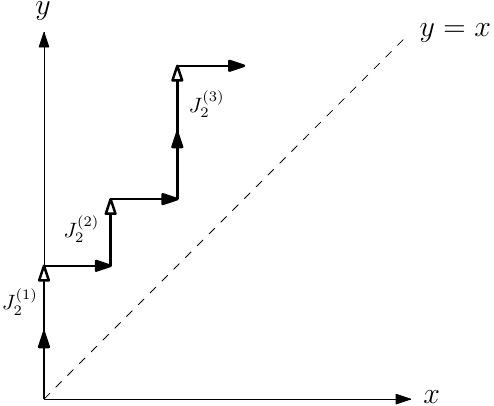}
\caption{A projected path of weight $90$ staying in $y\ge x$. Steps $I$ and $J_1$ have weight $1$. The $k$-th $J_2$-step ending at $(i,j)$ has weight $w_{i,j,k-1}=i+j+k-1$: $w(J_2^{(1)})=2$, $w(J_2^{(2)})=5$, $w(J_2^{(3)})=9$.}
  \label{fig:projectedpath}
\end{figure}

We are interested in two different classes of paths: 
First, the ones enumerated by $y_{n,n-k,k}$ are paths ending at $(n,n)$ that use $k$ steps $J_2$.
By the discussion above, the corresponding counting sequence is $b_{n,k} = y_{n,n-k,k}$.
Second, the ones enumerated by $y_{k,0,n}$ are paths ending at $(k,n)$ that use no steps $J_1$, and therefore only the steps $I$ and $J_2$. 
In particular, the weight of any step $J_2$ ending at $(i,j)$ is $w_{i,j,j} = i + 2j-1$. 
This class is enumerated by $c_{n,k} := y_{k,0,n}$.
Our strategy is now to prove Equation~\eqref{bceq} using these lattice paths.

One key observation here is that Equation~\eqref{bceq} has a very simple equivalent statement in terms of generating functions. 
Note that for $b_{n,k}$ we choose ordinary 
while for $c_{n,k}$ 
exponential generating functions. 
At the level of generating functions, the identity is equivalent to one between \emph{ordinary} and \emph{exponential} generating functions.

\begin{corollary}
We define the ordinary and shifted exponential generating functions
\begin{align} \label{eq:defAkBk}
B_k(z)=\sum_{n\ge 0} b_{n,k} z^{2n},
\qquad
C_k(z)=\sum_{n\ge k} c_{n,k}\frac{z^{\,n-k+1}}{(n-k+1)!}.
\end{align}
For $k \ge 1$, the Pons--Batle Identity~\eqref{eq:Pons-Batle} is equivalent to
\begin{align}\label{eq:relaakBk}
2B_k(z)=z^{2(k-1)} C_k(2z^2).
\end{align}
\end{corollary}

Our goal is now to inductively compute closed-form expressions for these generating functions that will allow to prove~\eqref{eq:relaakBk}.
We first discuss the case $k=0$, which has a slightly different form than \eqref{eq:relaakBk}, yet is well known and easy to prove.

\subsection{Case \texorpdfstring{$k=0$}{k=0} of \texorpdfstring{$B_k(z)$}{Bk(z)} and \texorpdfstring{$C_k(z)$}{Ck(z)}}
For $k=0$, the two-dimensional interpretation immediately shows that
\begin{align*}
b_{n,0}&=\frac{1}{n+1}\binom{2n}{n} &
\text{ and } &&
c_{n,0}&=(2n-1)!!,
\end{align*}
since all these paths are classical Dyck paths of length $2n$, counted by the Catalan numbers, and the weights satisfy $w_{0,k,k} = 2k-1$ (with $(-1)!!=1$).
From these closed forms, we directly see that the Pons--Batle Identity~\eqref{eq:Pons-Batle} holds for $k=0$:
\begin{align}
    \label{eq:a0nb0n}
\frac{b_{n,0}}{2^{n}}
= \frac{(2n)!}{(n+1)!\,n!\,2^{n}}
= \frac{(2n-1)!!}{(n+1)!}
= \frac{c_{n,0}}{(n+1)!}.
\end{align}
In terms of generating functions we have
$B_0(z)=\sum_{n\ge 0} b_{n,0} z^{2n}$ 
and
$C_0(z)=\sum_{n\ge 0} c_{n,0}\frac{z^{n+1}}{(n+1)!}.$
It is well-known that they possess the following closed forms:
\begin{align*}
B_0(z) &= \frac{1-\sqrt{1-4z^2}}{2z^2}&
\text{ and } &&
C_0(z) &= 1-\sqrt{1-2z}.
\end{align*}
Then the Pons--Batle Identity~\eqref{eq:Pons-Batle} for $k=0$ is equivalent to the easily verifiable equation
\[
2z^2B_0(z)=C_0(2z^2).
\]

\begin{remark}[Bijection on binary trees]\label{rem:bijectioncase0}
In this case, we can give a bijective interpretation of~\eqref{eq:a0nb0n} between plane binary trees (Catalan numbers) and leaf-labeled non-plane binary trees (odd double factorials), both with $n$ internal nodes and $n+1$ leaves.
The idea is to embed both structures into plane leaf-labeled binary trees. 
On the one hand, in plane trees the leaves are not labeled. Due to left-right order of the children there are $(n+1)!$ different ways to label the leaves.
On the other hand, in non-plane trees, we may choose a left-right-order of the children of each internal node, which can be done in $2^n$ many ways.
\end{remark}

\section{The generating function \texorpdfstring{$B_k(z)$}{Bk(z)} for \texorpdfstring{$k\geq 1$}{k>=1}} \label{GF of BK}
\newcommand{\Mup}{M_{up}}

We now focus on lattice paths associated with model $B$, enumerated by $B_k(z)$.
First, we transform it into the following bicolored Dyck path model (see Figure~\ref{fig:decomposition_Fk}): 
It consists of down steps $D = (1,-1)$ and two kinds of up steps, $U_1=(1,1)$ and $U_2=(1,1)$, corresponding to $I$, $J_1$, and $J_2$, respectively.
The weight of the steps $U_1$ and $D$ is $1$ each, while the weight of the step $U_2$ depends on the current trajectory: If this step is the $m$-th step in the path and it is the $k$-th step $U_2$ then its weight is 
\begin{align*}
    \hat{w}_{m,k} = m+k-1,
\end{align*}
which is a direct translation of the weight~\eqref{eq:wijk}.
We recall that the ordinary generating function of Dyck paths counted by length is
\begin{align*}
D(z):=B_0(z)=1+z^2+2z^4+5z^6+14z^8+\cdots,
\end{align*}
which is the unique power series solution of $D(z)=1+z^2D(z)^2$.
It is convenient to introduce
\begin{align} \label{eq:EshorthandDyck} 
E(z) = z D(z) = \frac{1-\sqrt{1-4z^2}}{2z}, \end{align} 
which is the unique power series solution of $E = z(1+E^2)$.

We will construct paths recursively with respect to the number $k$ of $U_2$-steps. 
We define the bivariate generating function 
\begin{align*}
    F_k(z,u) = \sum_{n,\ell \geq 0} f_{k,n,\ell} z^n u^{\ell},
\end{align*}
where $f_{k,n,\ell}$ is the number of bicolored Dyck meanders with $k$ colored $U_2$-steps such that the last step is a $U_2$-step that ends at level $\ell$ after $n$ steps.
As we will need this construction repeatedly, let $M_{\ell}(z)$ denote the generating function of Dyck meanders ending at level $\ell$.
By a standard last-passage decomposition, we have
\[
M_{\ell}(z)=D(z)\,E(z)^{\ell}.
\]

We further introduce the bivariate generating function $\Mup(z,u)$ for Dyck meanders
followed by an up step: the variable $z$ marks the length (number of steps) and
$u$ marks the final height. A meander ending at height $\ell$, followed by an up step, reaches height $\ell+1$ and contributes a factor $zu$.
Thus
\begin{align*}
    \Mup(z,u) 
        = \sum_{\ell \geq 0} M_{\ell}(z) z u^{\ell+1}
        = \sum_{\ell \geq 0} (uE(z))^{\ell+1}
        = \frac{uE(z)}{1-uE(z)}.
\end{align*}

\begin{figure}
  \centering
  \includegraphics[width=1\linewidth]{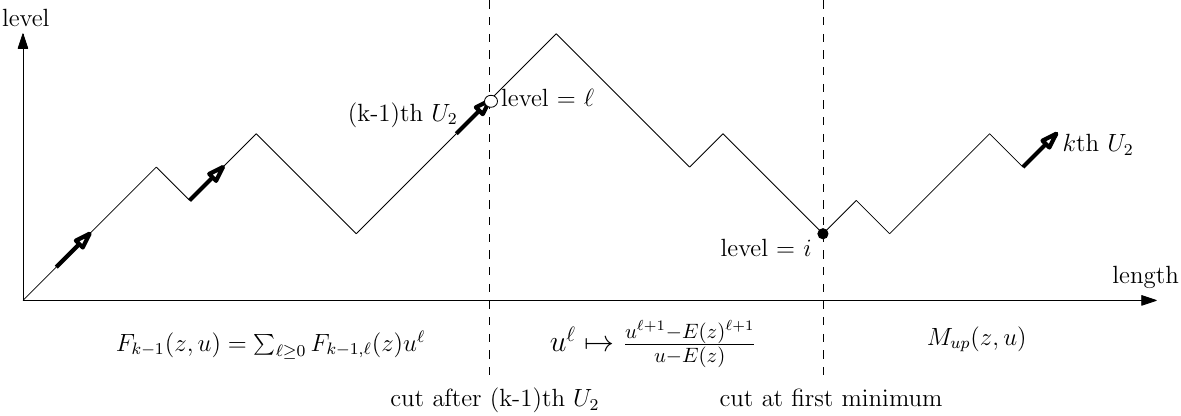}
\caption{Decomposition used to obtain $F_k(z,u)$ from $F_{k-1}(z,u)$. 
Cut after the $(k-1)$th $U_2$ step and at the first minimum between the $(k-1)$th and the $k$th $U_2$ step. 
The three parts correspond to $F_{k-1}(z,u)$, the operator $u^\ell \mapsto \frac{u^{\ell+1}-E(z)^{\ell+1}}{u-E(z)}$ (with $E(z)=zD(z)$), and a suffix $\Mup(z,u)$. 
The weight of the last $U_2$ step is implemented by applying $\frac{1}{z^{k-2}}\frac{d}{dz}z^{k-1}$ to the resulting generating function.}
  \label{fig:decomposition_Fk}
\end{figure}

With these ingredients we are ready to state our main result:
\begin{theorem}\label{theo:recursiveAk}
    The generating function $F_{k}(z,u)$ of bicolored Dyck meanders with $k$ colored $U_2$ steps ending with a $U_2$ step where $z$ marks the number of steps and $u$ the final level satisfies for $k \geq 1$
    \begin{align*}
        F_{k}(z,u) 
            &= \frac{1}{z^{k-2}}\frac{d}{dz} \left( z^{k-1} \frac{u F_{k-1}(z,u) -  E(z) F_{k-1}(z,E(z))}{u -  E(z)} \Mup(z,u) \right),
    \end{align*}
    with $F_0(z) = 1$. 
    The generating function $B_k(z)$ is then for $k \geq 0$ given by
    \begin{align*}
        B_k(z) 
            &= F_{k}(z, E(z)) D(z).
    \end{align*}
\end{theorem}

\begin{proof}
For $k=1$, a bicolored meander counted by $F_1(z,u)$ is a Dyck meander followed by an up step,
where this last up step is colored ($U_2$) and carries weight equal to the path length.
 This is obtained from $\Mup(z,u)$ by the standard pointing operator
$z\frac{d}{dz}$, hence
\begin{align}\label{eq:F1}
    F_1(z,u) 
        = z \frac{d}{dz} \Mup(z,u)
        = \frac{z uE'(z)}{(1-uE(z))^2}.
\end{align}

Assume $k\ge 1$. We compute $F_{k}(z,u)$ from $F_{k-1}(z,u)$ as follows.
Given a path with $k$ steps $U_2$, we cut the path into three parts at the following two points (see Figure~\ref{fig:decomposition_Fk}):
after the $(k-1)$th step~$U_2$, and at the first minimum between the $(k-1)$th and $k$th step~$U_2$.
Now we discuss each part separately. 
The first part is enumerated by $F_{k-1}(z,u)$, a path from the origin to a level $\ell$ marked by the variable $u$.
The second part starts at level $\ell$ and then runs to a level $0 \leq i \leq \ell$ ending with a down step if it is nonempty. 
We will discuss it below.
The third part is then a meander that starts from level $i$ and ends with a step $U_2$.
Therefore, the third part is modeled by a factor $\Mup(z,u)$ in which the weight of the last step has to be adjusted.
This weight is the length of the path plus $k-1$ which can be computed by the operator $\frac{1}{z^{k-2}}\frac{d}{dz} z^{k-1} (\cdot)$ applied to the full generating function. 
Indeed, for any monomial~$z^n$, 
\begin{align*}
\frac{1}{z^{k-2}}\frac{d}{dz}\!\left(z^{k-1}z^n\right) = (n+k-1)\,z^n,
\end{align*}
so the operator multiplies each contribution by the path length plus $k-1$.

For the second part, we use a last-passage-decomposition by cutting each time the path attains a new lower level for the first time.
Between levels the path is a classical Dyck path associated with $D(z)$ and the first step to a lower level is a step $D$ associated with $z$. 
Thus, a path from level $\ell$ to level $i$ corresponds to $(zD(z))^{\ell-i}$.
Finally, we sum over all possible levels $0 \leq i \leq \ell$ of the minimum.
In terms of operators, we replace the weight $u^{\ell}$ of a path ending at level $\ell$ as follows (recall that $E(z) = z D(z)$):
\begin{align*}
    u^{\ell} 
        &\mapsto u^{\ell} + E(z)u^{\ell-1} + E(z)^2 u^{\ell-2} + \dots + E(z)^{\ell} 
        =  \frac{u^{\ell+1} -  E(z)^{\ell+1}}{u -  E(z)}.
\end{align*}
Finally, combining all three parts, this gives the claimed recurrence for $F_{k}(z,u)$.
In order to compute $B_{k}(z)$ we perform another first-passage decomposition at the end to let the walk run from altitude $\ell$ back to $0$. 
In terms of generating functions this gives
$
    B_k(z) = F_{k}(z, E(z)) D(z),
$
where the final factor $D(z)$ corresponds to a possibly empty Dyck path at the end of the walk.
Note that setting $F_{0}(z,u)=1$ this also gives the correct result for $k=0$ and $k=1$.
\end{proof}

We compute $F_k$ and $B_k$ efficiently by expressing them in $E$, using the algebraic equation $E = z(1+E^2)$. In particular, it allows us to express $z$ in terms of $E$ such that 
$
    z = \frac{E}{1+E^2}.
$
Thereby, we are able to compute $B_k(z)$ directly on our computer for all $k \leq 250$ using Maple. 
The first generating functions are then given by
\begin{align*}
	B_0(z) &= 1+E^2 = \frac{E}{z} = D = 1+{z}^{2}+2\,{z}^{4}+5\,{z}^{6}+O( {z}^{8}) , & (\oeis{A000108})
           \\
	B_1(z) &= \frac{E^2(1+E^2)^2}{(1-E^2)^3} = {z}^{2}+7\,{z}^{4}+38\,{z}^{6}+O( {z}^{8} ) ,& (\oeis{A000531})
                 \\
	B_2(z) &= \frac{E^4(1+E^2)^2 (8 E^4+15E^2+7)}{(1-E^2)^7} = 7\,{z}^{4}+106\,{z}^{6}+1010\,{z}^{8}+O( {z}^{10} ). &
\end{align*}
Using Theorem~\ref{theo:recursiveAk}, one can show that all $B_k \in \mathbb{Q}(z, E)$ for all $k \geq 0$, where $\mathbb{Q}(z,E)$ is an algebraic extension of degree $2$ of the rational function field $\mathbb{Q}(z)$. Hence, all $B_k$ are algebraic functions of degree $2$.

\section{The generating function \texorpdfstring{$C_k(z)$}{Ck(z)} for \texorpdfstring{$k \geq 1$}{k>=1}}

Recall that $c_{n,k}$ counts paths from $(0,0)$ to $(k,n)$ using steps $I=(1,0)$ and $J_2=(0,1)$ that never cross the line $y=x$.
Hence, the initial conditions $c_{n,k}=0$ for $n<k$ follow immediately.
Now, we remove again the last step from the path.
This gives the recurrence for $n \geq k$
\begin{align*}
    c_{n,k} = c_{n,k-1} + (2n+k-1)c_{n-1,k},
\end{align*}
since the weight~\eqref{eq:wijk} of the step $J_2$ is $w_{k,n,n}=2n+k-1$.
Recall that $C_k(z)$ denotes the shifted exponential generating function defined in~\eqref{eq:defAkBk}.
For $k\ge 1$, $C_k(z)$ satisfies the differential equation with initial condition $C_k(0)=0$:
\begin{align}   \label{eq:odeBbar}
(1-2z)C_k'(z)-(3k-1)C_k(z)=C_{k-1}''(z).
\end{align}

For $k=0$ we have
$
    C_0(z) = 1 - \sqrt{1-2z},
$
which satisfies $C_0(0)=0$.
Note that for $k=0$ the differential equation~\eqref{eq:odeBbar} is also satisfied if we set
${C}_{-1}''(z) = 1$, i.e., for instance ${C}_{-1}(z) = \frac{z^2}{2}$.

\begin{proposition}\label{prop:BkODE}
For $k \ge 1$, the generating function $C_k(z)$ satisfies
\begin{align}
\label{eq:barbkopgf}
\prod_{j=3k-1}^{4k-1}(T-j){C}_k(z)=0,   
\end{align}
where the operator $T$ is defined by $T=(1-2z) \frac{d}{dz}$. Moreover, it can be written as
\begin{align}
\label{eq:barbkgf}
{C}_k(z)=\sum_{i=0}^{k} \gamma_{i,k} (1-2z)^{-\frac{i+3k-1}{2}},    
\end{align}
where 
\begin{align*}
\gamma_{i,k} = \frac{(i+3k-3)(i+3k-5)}{i}\gamma_{i-1,k-1} \qquad \text{for $1 \leq i \leq k$}
\end{align*}
with the initial terms $\gamma_{0,1}=-1$ and $\gamma_{1,1}=1$,
and $\gamma_{0,k} = -\sum^{k}_{i=1}\gamma_{i,k}$.
\end{proposition}

Recall that in Identity~\eqref{eq:relaakBk} we need to compute $C_k(2z^2)$ which therefore consists of powers of $\sqrt{1-4z^2}$. Using the algebraic equation $E = z(1+E^2)$ this square root satisfies 
\[
  \sqrt{1-4z^2} = \frac{1-E^2}{1+E^2},
\]
which allows us to verify the identity for computed values of $k$ by checking whether the two rational functions (in $E$) for $B_k$ and $C_k$ coincide.
However, currently we do not know how to prove this result for arbitrary $k$.

The first few terms of $C_k(2z^2)$ are given by
\begin{align*}
	C_0(2z^2) &= 1-(1-4z^2)^{1/2} \\
	       &= 2\,{z}^{2}+2\,{z}^{4}+4\,{z}^{6}+10\,{z}^{8}+O ( {z}^{10}), 
            & (\oeis{A284016})
           \\
	C_1(2z^2) &= (1-4z^2)^{-3/2}-(1-4z^2)^{-1} \\
	       &= 2\,{z}^{2}+14\,{z}^{4}+76\,{z}^{6}+374\,{z}^{8}+O( {z}^{10}),
           & (\oeis{A172060}) 
           \\
	C_2(2z^2) &= \frac{15}{2} (1-4z^2)^{-7/2} - 8 (1-4z^2)^{-3} + \frac{1}{2} (1-4z^2)^{-5/2}\\
	       &= 14\,{z}^{2}+212\,{z}^{4}+2020\,{z}^{6}+15480\,{z}^{8}+O ( {z}^{10}).
\end{align*}
Using Maple, we can easily check that the identity holds for $k \le 250$.

\section{Limit laws of Young tableaux with walls and \texorpdfstring{$k=1$}{k=1} extra cell}

Except for $k=0$ (see Remark~\ref{rem:bijectioncase0}), no bijection is known between the objects enumerated by $b_{n,k}$ and $c_{n,k}$. We begin by analyzing parameters of Young tableaux with walls for $k=1$.
In this section, we study two parameters for
$b_{n,1}$ (i.e., tableaux of shape $y_{n,n-1,1}$) and one parameter for
$c_{n,1}$ (i.e., tableaux of shape $y_{1,0,n}$).

\subsection{Position of the unique cell in the bottom row in \texorpdfstring{$\Bc_{n,1}^*$}{B*\_{n,1}}}

A first natural parameter in tableaux of the class $\Bc^*_{n,1}$ is the horizontal position or column number of the unique cell in the bottom row.
These tableaux of width $n$ are enumerated by $b_{n,1} = y_{n,n-1,1}$.
Under the bijection with bicolored Dyck paths (Sections~\ref{sec:conjinlps},~\ref{GF of BK}), this bottom cell and its middle-row neighbor correspond to the unique $U_2$ up-step; all other middle-row cells correspond to $U_1$ steps.
Therefore, the bottom cell is in column $m$ if and only if the $U_2$-step is the $m$-th up step in the Dyck path.

Let $G(z,x)=\sum_{n,m\ge 0} g_{n,m} z^{n}x^{m}$, where $g_{n,m}$ counts Dyck paths of length $n$ with the $U_2$ step at the $m$-th up-step. Clearly, $G(z,1)=A_1(z)$.

\begin{proposition}\label{prop:GFGzxBcol}
    The generating function $G(z,x)$ is algebraic of degree $4$ given by
    \begin{align*}
        G(z,x) &= 2\,{\frac {1-\sqrt {1-4\,x{z}^{2}}}{\sqrt {1-4\,x{z}^{2}} 
                    \left( \sqrt {1-4\,{z}^{2}}+\sqrt {1-4\,x{z}^{2}} \right) ^{2}}}
               \\&= 
               x{z}^{2}+ \left( 5\,{x}^{2}+2\,x \right) {z}^{4}+ \left( 21\,{x}^{3}+
                12\,{x}^{2}+5\,x \right) {z}^{6}+ O ( {z}^{8} ) 
            .
    \end{align*}
\end{proposition}

\begin{proof}
    We revisit the construction of $A_1(z)$ and $F_1(z,u)$ in \eqref{eq:F1} to track the number of up-steps occurring up to and including the unique $U_2$-step. We decompose such a path at this $U_2$-step: the prefix (represented by $M_{up}(z,u)$) is a Dyck meander ending at level $\ell$ with a $U_2$-step, followed by a suffix that returns from level $\ell$ to $0$.

    Suppose the prefix has length $n$, ends at level $\ell$, and contains $m$ up-steps. Since the number of down-steps is $m-\ell$, the total length is $n = m + (m-\ell) = 2m-\ell$, which implies $m = \frac{n+\ell}{2}$. To track these up-steps, we apply the substitution $z \to zx^{1/2}$ and $u \to ux^{1/2}$ to $M_{up}(z,u)$. 

    The $U_2$-step is accounted for by the operator $z \partial_z$ to fix its weight. Finally, substituting $u \to zD(z)$ for the return path and multiplying by $D(z)$ yields $G(z,x)$, where elementary simplification provides the final result.
\end{proof}

This generating function allows us to define the random variable $X_n$ of the column index of the unique bottom cell for a tableaux from $\Bc^*_{n,1}$ drawn uniformly at random as follows:
\begin{align*}
    \mathbb{P}(X_n=m) = \frac{[z^{2n} x^m]G(z,x)}{[z^{2n}]G(z,1)}.
\end{align*}

The limit distribution of $X_n$ is given below; on average, the unique cell lies at column $\frac{2}{3}n$ with standard deviation $\frac{\sqrt{2}}{6} n \approx 0.236 n$. For related asymptotic results, see~\cite{Durand2004,Banderieretal2012Nongaussian}.

\begin{theorem}
    The random variable $X_n$ of the distribution of the column index of the unique cell in the bottom row in $\mathcal{B}^*_{n,1}$ converges after rescaling by $n$ in law with convergence of all moments to a random variable $X$ distributed like a $\operatorname{Beta}(2,1)$ distribution:
    \begin{align*}
        \frac{X_n}{n} &\stackrel{d}{\to} X  \qquad \text{ where } \qquad
        X \stackrel{d}{=}\operatorname{Beta}(2,1).
    \end{align*}
    The probability density function of $X$ is given by $f_X(x) = 2x$ for $x \in [0,1]$ and is equal to $0$ otherwise. 
    The moments are given by $\E(X^r) = \frac{2}{r+2}$.
\end{theorem}

\newcommand{\tG}{\tilde{G}}
\begin{proof}[Proof (Sketch)]
Our strategy is to show that the asymptotic moments of $X_n$ coincide with the moments of a $\mathrm{Beta}(2,1)$ distribution. Since the Beta distribution is uniquely characterized by its moments on bounded intervals (the Hausdorff moment problem), the convergence in law follows from the Fr{\'e}chet--Shohat Theorem~\cite{FrechetShohat1931Moments}; see also, e.g.,~\cite{BanderierEtal2024ML} for this approach.

By standard principles, the $r$-th factorial moment $\E((X_n)_r)$ is obtained by extracting the coefficients of the $r$-th partial derivative of the generating function $G(z,x)$ with respect to $x$, evaluated at $x=1$. To extract the asymptotics for large $n$, we apply singularity analysis~\cite{FlajoletSedgewick2009AC}.

Setting $t=4z^2$, we isolate the dominant singular behavior of $G(\sqrt{t}/2,x)$ near $t=1$. Discarding constants and non-singular factors that do not affect the asymptotic distribution, the core shape reduces to:
\[
    \tG(x) \equiv \tG(t,x) = \frac{1}{g(x)(g(1)+g(x))^2}, \quad \text{where} \quad g(x) = \sqrt{1-xt}.
\]
The asymptotic moments depend on the derivatives $\tG^{(r)}(1)$. 
By induction, we see that the generic shape of the $r$-th derivative is
\begin{align*}
    \tG^{(r)}(x) &= t^r \sum_{i=0}^r \frac{m_{r,i}}{g(x)^{2r+1-i} (g(1)+g(x)))^{i+2}}
    & \text{ and } &&
    \tG^{(r)}(1) &= \frac{t^r}{4 g(1)^{2r+3}} \sum_{i=0}^r \frac{m_{r,i}}{2^{i}}.
\end{align*}
Therefore, we need to study the sum $\sum_{i=0}^r \frac{m_{r,i}}{2^{i}}$. 
For this purpose, we define a polynomial associated with this sum by
$
    m_r(u) = \sum_{i=0}^r m_{r,i} u^i,
$
such that we are interested in $m_r(1/2)$.
The associated exponential generating function $M(u,y) = \sum_{r \ge 0} m_r(u) \frac{y^r}{r!}$ satisfies the following linear partial differential equation:
\[
    u(u+1) \frac{\partial M}{\partial u} + 2(y-1)\frac{\partial M}{\partial y} + (2u+1) M = 0, \quad \text{with } M(u,0)=1.
\]
This equation admits a closed-form solution, which then yields the explicit asymptotic growth of the derivatives. Finally, we obtain the asymptotic factorial moments:
\[
    \E((X_n)_r) \sim \frac{2}{r+2} n^r.
\]
Because these factorial moments are asymptotically equivalent to the raw moments, we have $\E(X_n^r) \sim \frac{2}{r+2} n^r$. These are precisely the scaled moments of a $\mathrm{Beta}(2,1)$ distribution.
\end{proof}

\begin{figure}
    \centering
    \includegraphics[width=0.28\linewidth]{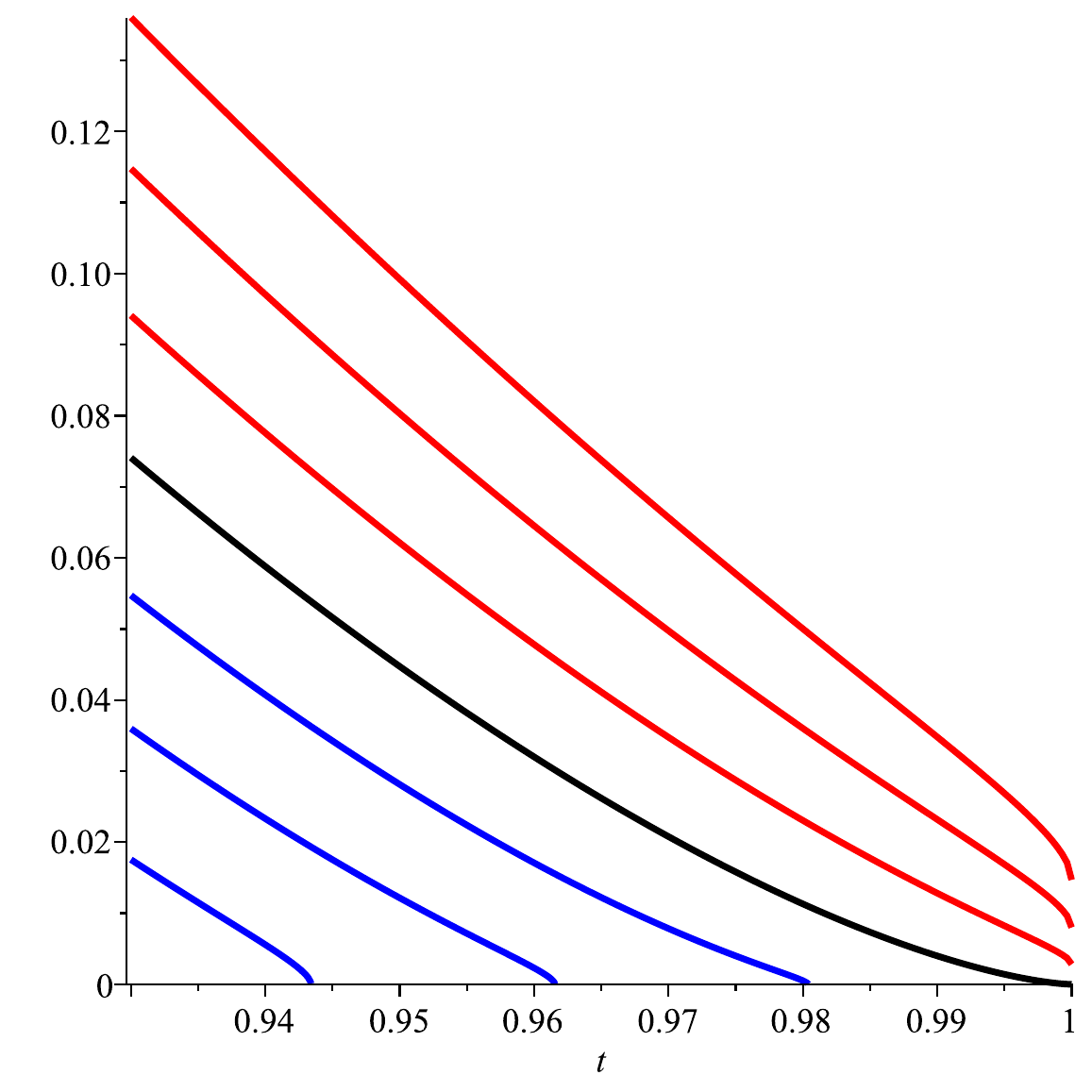}
    \qquad
    \includegraphics[width=0.28\linewidth]{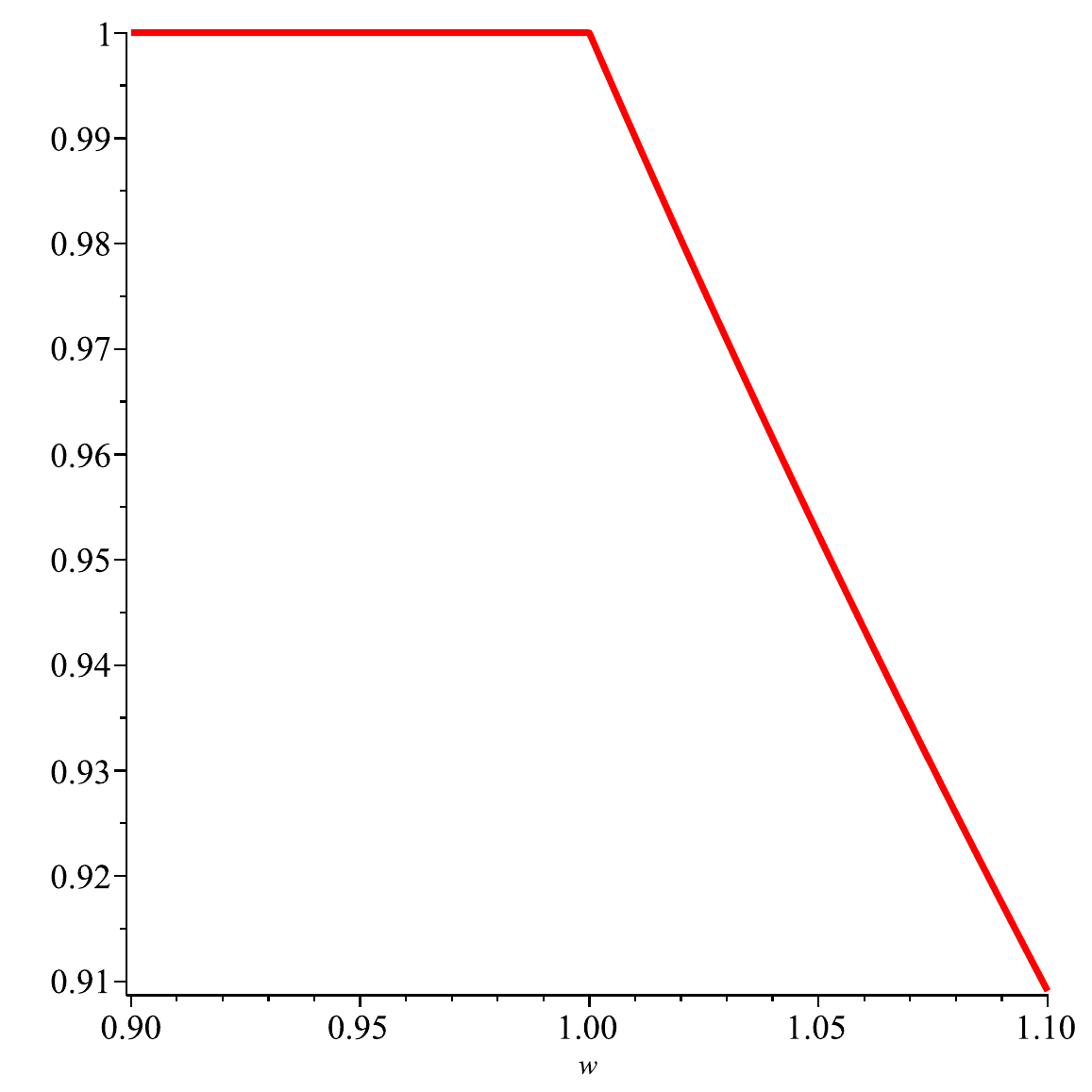}
    \qquad
    \includegraphics[width=0.28\linewidth]{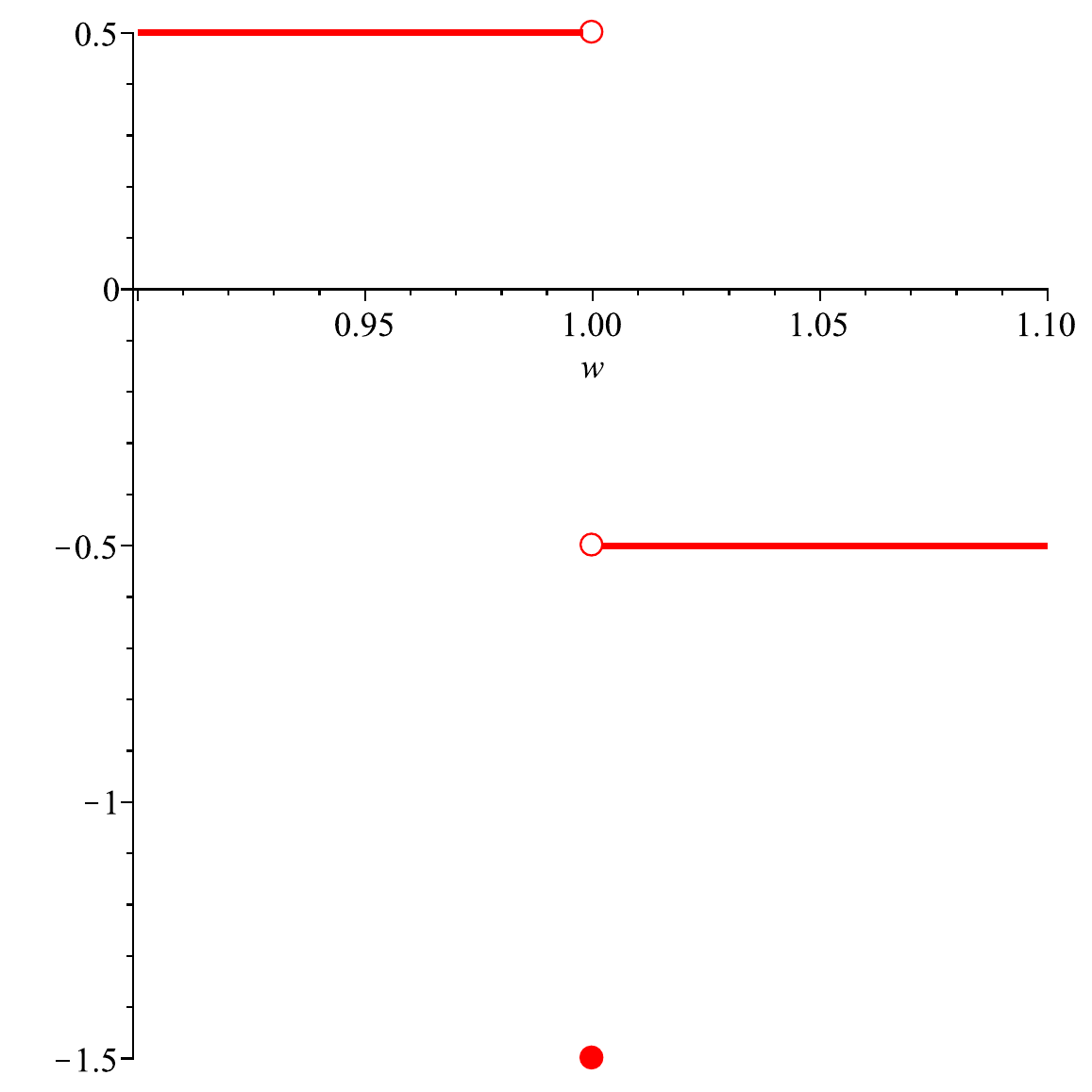}
    \qquad
    \caption{Left: A plot of $t \mapsto 1/\tG(t,x)$ for $t \in [0.9,1]$ when $x \in [0.9,1.1]$ (red curves $x<1$, black curve $x=1$; blue curves $x>1$);
    center: the singularity $\rho(x)$; right: the dominant singular exponent $\alpha(u)$ in $(1-t/\rho(x))^{\alpha(x)}$.
    Both, the singularity $\rho(x)$ and the singular exponent $\alpha(x)$ experience a non-analytic transition at $x=1$.
     This is due to a collapsing of two singular terms at $x=1$.}
    \label{fig:Gtilde_phasetransition}
\end{figure}

\subsection{Value in the unique cell in the bottom row in \texorpdfstring{$\Bc_{n,1}^*$}{B*\_{n,1}}}

A second natural parameter in the tableaux of the class $\Bc^*_{n,1}$ is the value in the unique cell in the bottom row.
As in the previous case, we use the bijection between the tableaux and bicolored Dyck paths described in Sections~\ref{sec:conjinlps} and~\ref{GF of BK}.
Recall, that this cell is associated with the unique up step $U_2$.
Revisiting the bijection we observe that if $U_2$ is the $k$-th step in the Dyck path, then the value in this cell may be any integer of $\{1,2,\dots,k\}$, which corresponds to the weight of this step in the recurrence. 
These paths can be visualized as enriched Dyck paths in which we attach an additional pointer to this special step $U_2$.
This pointer starts at this $U_2$-step and ends at any earlier step or the $U_2$-step itself.
In the tableaux, the value of the cell is then $k$ if the pointer ends at the $k$-th step from the beginning.

We mark this value with the variable $x$ and define the bivariate generating function
\[
H(t,x)=\sum_{n,m\ge 0} h_{n,m}\,t^{n}x^{m},
\]
where $h_{n,m}$ is the number of Dyck paths of length $2n$ in which the pointer of the $U_2$-step ends at the $m$-th step from the beginning.
Note that $H(z^2,1)=A_1(z)$.

\begin{proposition}
    The generating function $H(t,x)$ is algebraic of degree $4$ given by
    \begin{align*}
        H(t,x) &= \frac {x \left( 1-\sqrt {1-4t} \right)  \left( 1-x + x\sqrt {1-4t}-\sqrt {1-4tx^2} \right) }{ 
        {(1-x)\sqrt {1-4t}}  \left( 1-4tx + \sqrt {1-4tx^2}\sqrt {1-4t}-\sqrt {1-4t} - \sqrt {1-4tx^2} \right) 
        }
               \\&= 
               xt+ \left( {x}^{2}+2\,x+4 \right) x{t}^{2}+ \left( 2\,{x}^{4}+4\,{x}^
{3}+7\,{x}^{2}+10\,x+15 \right) x{t}^{3}+O \left( {t}^{4} \right) 
            .
    \end{align*}
\end{proposition}

Now, we are able to define the random variable $Y_n$ associated with the value in the unique bottom cell of a tableaux from $\Bc^*_{n,1}$ drawn uniformly at random:
\begin{align*}
    \mathbb{P}(Y_n=m) = \frac{[t^{n} x^m]H(t,x)}{[t^{n}]H(t,1)}.
\end{align*}

\begin{theorem}\label{theo:Yn_valueBn1}
    The random variable $Y_n$ of the distribution of the value of the unique bottom cell in $\mathcal{B}^*_{n,1}$ converges after rescaling by $n$ in law with convergence of all moments to a random variable $Y$ distributed like a $\operatorname{Beta}(1,2)$ distribution:
    \begin{align*}
        \frac{Y_n}{2n} &\stackrel{d}{\to} Y  \qquad \text{ where } \qquad
        Y \stackrel{d}{=}\operatorname{Beta}(1,2).
    \end{align*}
    The probability density function of $Y$ is given by $f_Y(y) = 2(1-y)$ for $y \in [0,1]$ and is equal to $0$ otherwise. 
    The moments are given by $\E(Y^r) = \frac{2}{(r+1)(r+2)}$.
\end{theorem}

\begin{proof}[Proof (Sketch)]
    We follow the approaches~\cite{DrmotaSoria1997Images,Wallner2020Halfnormal} that use Lévy's continuity theorem for characteristic functions (but for different distributions).\footnote{Note that this approach does not provide convergence of moments.} 
    It says, that if the characteristic functions $\phi_n(u) := \E(e^{i u Y_n})$ converge pointwise for all real numbers $t \in \mathbb{R}$ against the characteristic function $\phi(u) := \E(e^{i u Y})$ then $Y_n$ converges in distribution against $Y$. 
    A straightforward computation shows that $\phi(u) = \frac{2(iu + 1 - e^{iu})}{u^2}$ for the $\operatorname{Beta}(1,2)$-distribution.

    Now we compute $\phi_n(u)$.
    Using the generating function we can express it as 
    \begin{align*}
        \phi_n(u) = \frac{[t^n] H(t, e^{i u/n}) }{[t^n] H(t,1)}.
    \end{align*}
    What follows now are further technical computations that we omit. 
    It remains to use Cauchy's coefficient formula, expand around the singularity, and use a Hankel contour to expand the coefficients. Finally, using the Hankel representation of the Gamma function the result is obtained and the claim follows.
\end{proof}

\newcommand{\cp}{v}
\newcommand{\CP}{V}
\subsection{Value in the unique cell in the top row in \texorpdfstring{$\Cc_{n,1}^*$}{C*\_{n,1}}}

We now consider the analogous parameter in $\Cc^*_{n,1}$, enumerated by $y_{1,0,n}$. The top row has a single cell (column 1), the middle and bottom rows each have $n$ cells separated by walls. Thus, the only nontrivial parameter is the value in the top cell.

Contrary to the previous section, our approach consists now in enriching the underlying recurrence for $y_{1,0,n}$ by another parameter and not the associated generating function. 
Let $\cp_{n,i}$ for $n \geq 0 $ and $1\le i\le 2n+1$ denote the number of Young tableaux of
shape $\Cc_{n,1}^*$ whose unique top cell contains the value $i$. 

\begin{lemma}
The sequence $\cp_{n,i}$ satisfies the following recurrence:
\begin{align} \label{eq:recnumy10n}
\cp_{n,i}=
\begin{cases}
(2n-i)\,\cp_{n-1,i} + (i-1)\,\cp_{n-1,i-1}, & \text{ for } 1\le i\le 2n,\\
(2n-1)!!, & \text{ for } i=2n+1,
\end{cases}
\qquad n\ge 1,
\end{align}
where
$\cp_{n,i}=0$ outside of the above-defined domain.  
\end{lemma}

\begin{proof}[Proof (Sketch)]
The recurrence is obtained by removing the largest entry $2n+1$.
There are two possible positions: either it is in the unique cell in the top row, or in the last cell in the middle row. 
\end{proof}

We will now analyze this recurrence using a mixed exponential generating function.
Let 
\[
    \CP(t,x) = \sum_{n,i\ge0} \cp_{n,i} \, \frac{t^n}{n!}x^i .
\]

\begin{proposition}
    The generating function $\CP(t,x)$ is algebraic of degree $4$ given by
    \begin{align*}
        \CP(t,x) &= {\frac {{x}^{2}}{1-x} \left( {\frac {1-x
-x\sqrt {1-2t}}{1-2x+2t{x}^{2}}  }-{\frac {1}{\sqrt {1-2t{x}^{2}}}}
 \right) }
               \\&= 
               {x}^{3}t+ \left( 3x^5 + 3x^4 + x^3 \right) \frac{{t}^{2}}{2}+ \left( 15x^7 + 15x^6 + 15x^5 + 9x^4 + 3x^3 \right) \frac{{t}^{3}}{6}+ O\left( {z}^{4} \right) 
            .
    \end{align*}    
\end{proposition}

Let $Z_n$ be the random variable equal to the value in the unique top cell of a uniformly random standard Young tableau from $\Cc^*_{n,1}$:
\begin{align*}
    \mathbb{P}(Z_n=m) = \frac{[t^{n} x^m]V(t,x)}{[t^{n}]V(t,1)}.
\end{align*}
Then from the generating function we directly get
\[
\mathbb{E}[Z_n]
=
\frac{[t^n]\,\partial_x \CP(t,1)}{[t^n]\,\CP(t,1)}=\frac{\cp'_n(1)}{\cp_n(1)}
=
n-\frac{\sqrt{\pi n}}{2}+O(1).
\]
One can get all asymptotic moments using the same strategy as before.
In this case, the limiting distribution is a $\operatorname{Beta}(1,1)$ law, which is the uniform distribution on $[0,1]$.
We omit its proof which follows the same lines as the one of Theorem~\ref{theo:Yn_valueBn1}.
\begin{theorem}   
The random variable $Z_n$ of the distribution of the value of the unique top cell in $\mathcal{C}^*_{n,1}$ converges after rescaling by $n$ in law to a random variable $Z$ distributed like a uniform distribution on $[0,1]$: 
\begin{align*}
    \frac{Z_n}{2n}  &\stackrel{d}{\to} Z \qquad \text{ where } \qquad
    Z \stackrel{d}{=}\operatorname{U}(0,1).
\end{align*}
The probability density function of $Z$ is given by $f_Z(z) = 1$ for $z \in [0,1]$ and is equal to $0$ otherwise. 
The moments are given by $\E(Z^r) = \frac{1}{r+1}$.
\end{theorem}

\section{Conclusion and Outlook}
\label{sec:conclusion}
We have developed a combinatorial and analytic framework translating the Pons--Batle identity into a system of algebraic generating functions derived from decorated Dyck paths and Young tableaux with walls. Our approach not only verifies the identity for finite $k \le 250$ but also uncovers the probabilistic behavior of tree-child networks, as illustrated by the exact Beta and Uniform limit laws for $k=1$. These results reveal a general mechanism in which the coalescence of singular terms at the dominant singularity produces non-analytic transitions leading to Beta-type limit distributions.

Looking forward, several directions emerge naturally. First, we aim to abstract the mechanism behind the Beta distributions and formulate general conditions on algebraic or D-finite generating functions that systematically produce $\mathrm{Beta}(\alpha, \beta)$ limit laws. Second, we plan to investigate the joint distribution of structural parameters and their correlations, and to extend our probabilistic analysis to arbitrary fixed $k>1$, including the sum of values in the bottom row of $k$ cells, which could provide deeper insights into the structure of Pons--Batle word classes. Finally, the flexibility of the decorated Dyck path model suggests its applicability beyond standard tree-child networks; we intend to extend this framework to $d$-combining networks, offering a unified analytic approach to their exact enumeration and asymptotic properties.

\bibliography{bibliography}

\clearpage
\appendix

\section{Proof of Proposition~\ref{prop:BkODE}}

First, we define two operators:
\[D=\frac{d}{dz} \qquad \text{ and } \qquad  T=(1-2z)D, \]
which act linearly on formal power series. 
Now, Equation~\eqref{eq:odeBbar} can  be  rewritten  into
\begin{align}
\label{eq:barbkop}
(T-(3k-1)){C}_k(z)=D^2{C}_{k-1}(z).
\end{align}

From the definitions it follows directly that $D^2T=TD^2-4D^2=(T-4)D^2$.
Let $P(X)$ be any polynomial in the formal variable~$X$. Then, inductively, it holds that
\begin{align}
\label{eq:opswitch}
D^2 P(T)=P(T-4) D^2.    
\end{align}
Now, we turn to the claimed factored differential equation~\eqref{eq:barbkopgf} and we start with the case $k=0$.
From~\eqref{eq:odeBbar} we see that
\[(1-2z) {C}_0'(z)+{C}_0(z)=1, \] whose second derivative is 
\[(1-2z) {C}_0'''(z)-3{C}''_0(z)=(T-3){C}''_0(z)=0. \]
When $k=1$, we get from~\eqref{eq:odeBbar} that 
\begin{align*}
(1-2z) {C}_1'(z)-2{C}_1(z)=(T-2){C}_1(z)={C}''_0(z).
\end{align*}
After multiplying by $(T-3)$ from the left we get
\begin{align*}
(T-3)(T-2){C}_1(z)
=(T-3){C}''_0(z)=0.    
\end{align*}
This proves the claim for $k=1$.
Now we use induction to prove \eqref{eq:barbkopgf} for arbitrary $k>1$.
Assume that the claim holds for $k-1$: 
\[\prod_{j=3k-4}^{4k-5}(T-j){C}_{k-1}(z)=Q_{k-1}(T) \,{C}_{k-1}(z)=0, \]
where $Q_{k-1}(T)$ is a polynomial in $T$. 
The main idea now is to repeatedly use the relation~\eqref{eq:opswitch} together with the operator identity~\eqref{eq:barbkop} as follows
\begin{align*}
0&=D^2 0=D^2((Q_{k-1}(X)\big|_{X=T}) \,{C}_{k-1}(z))=D^2(Q_{k-1}(X)\big|_{X=T}) \,{C}_{k-1}(z) \\&=(Q_{k-1}(X) \,\big|_{X=T-4}) \,(D^2 {C}_{k-1})=(Q_{k-1}(X) \,\big|_{X=T-4}) \, ((T-(3k-1)){C}_k) \\&=\prod_{j=3k}^{4k-1}(T-(3k-1)){C}_k =\prod_{j=3k-1}^{4k-1}(T-j) {C}_k \\&=(Q_{k}(X)\big|_{X=T}) {C}_k,     
\end{align*}
Thus we get (\ref{eq:barbkopgf}).

For any $\phi_j$ such that $\phi_j=(1-2z)^\alpha$, we have $T\phi_j=-2\alpha\phi_j$, thus $\phi_j$ is the eigenfunction of the operator $T$. Note that $ {C}_k(z)$ can be written as the linear combination  of $\phi_j$ ($3k-1 \leq j \leq 4k-1$) such that $\phi_j$ is the  solution of 
$(T-j) \phi_j=(-2\alpha-j)\phi_j=0,$     which implies that  \[ {C}_k(z)=\sum_{j=3k-1}^{4k-1}c'_{j,k} (1-2z)^{-\frac{j}{2}}. \] 
Substitute the equation above in \eqref{eq:barbkop}. For the left side, we have
\begin{align*}
(T-(3k-1))\bar{B_k}
  &= (T-(3k-1))\sum_{j=3k-1}^{4k-1}c'_{j,k} (1-2z)^{-\tfrac{j}{2}} \\
  &= \sum_{j=3k-1}^{4k-1} (j-(3k-1))c'_{j,k} (1-2z)^{-\tfrac{j}{2}},
\end{align*}
and for the right side we have
\begin{align*}
D^2B_{k-1}
  &= D^2\!\left(\sum_{n=3k-4}^{4k-5}c'_{n,k-1} (1-2z)^{-\tfrac{n}{2}}\right) \\
  &= \sum_{n=3k-4}^{4k-5} n(n+2)c'_{n,k-1} (1-2z)^{-\tfrac{n+4}{2}} \\
  &= \sum_{n=3k}^{4k-1} (n-4)(n-2)c'_{n-4,k-1} (1-2z)^{-\tfrac{n}{2}}.
\end{align*}
Thus, for $3k \leq j \leq 4k-1$, we have
\[
c'_{j,k} = \frac{(j-2)(j-4)}{j-(3k-1)}c'_{j-4,k-1}.
\]

For $j=3k-1$, the coefficient in front of $c'_{3k-1,k}$ in $(T-(3k-1))\bar{B_k}$ vanishes, so $c'_{3k-1,k}$ cannot be determined from this relation alone. However, since
\[
{C}_k(0)=\sum_{j=3k-1}^{4k-1}c'_{j,k}=0,
\]
we know that $c'_{3k-1,k} = - \displaystyle \sum_{j=3k}^{4k-1} c'_{j,k}.$

For the first terms when $k=1$, since \[(T-2)(T-3) {C}_1(z)=0,\] we have \[B_1(z)=c'_{2,1}(1-2z)^ {-1}+c'_{3,1}(1-2z)^{-\frac{3}{2}}.\] Furthermore, from \eqref{eq:barbkop},
\begin{align*}
(T-2){C}_1(z)&=(T-2)c'_{2,1}(1-2z)^{-1} + (T-2)c'_{3,1}(1-2z)^{-\tfrac{3}{2}}\\&=c'_{2,1}(T-2)(1-2z)^{-1}+ c'_{3,1} (T-2)(1-2z)^{-\tfrac{3}{2}} \\&
  = c'_{2,1}(2-2)(1-2z)^{-1} + c'_{3,1}(3-2)(1-2z)^{-\tfrac{3}{2}} \\
  &= c'_{3,1}(1-2z)^{-\tfrac{3}{2}}.
\end{align*}

On the other hand,
\[
D^2{C}_0(z) = (1-2z)^{-\tfrac{3}{2}}.
\]
Comparing both sides gives \(c'_{3,1}=1\). Using ${C}_1(0)=0$, we obtain $c'_{2,1}=-1$ as well. Finally, substitute $j$ by $i+3k-1$, we arrive at~(\ref{eq:barbkgf}).
\qed

\pagebreak

\section{Some examples}\label{apb}

\begin{figure}[H]
    \centering
    \includegraphics[scale=1.7]{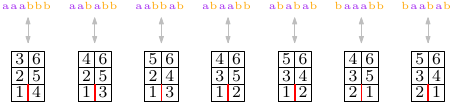}
    \caption{All words of the set $\mathcal{A}_2$ and their corresponding Young tableaux in $\mathcal{A}^*_2$.}
    \label{a2bijection}
\end{figure}

\begin{figure}[H]
    \centering
    \includegraphics[scale=1.7]{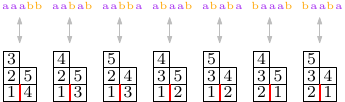}
    \caption{All words of the set $\mathcal{C}_{2,1}$ and their corresponding Young tableaux in $\mathcal{C}^*_{2,1}$.}
    \label{c21bijection}
\end{figure}

\begin{figure}[H]
    \centering
    \includegraphics[scale=1.7]{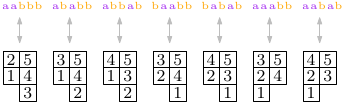}
    \caption{All words of the set $\mathcal{B}_{2,1}$ and their corresponding Young tableaux in $\mathcal{B}^*_{2,1}$.}
    \label{b21bijection}
\end{figure}

\begin{figure}[h]
    \centering
    \includegraphics[scale=1.1]{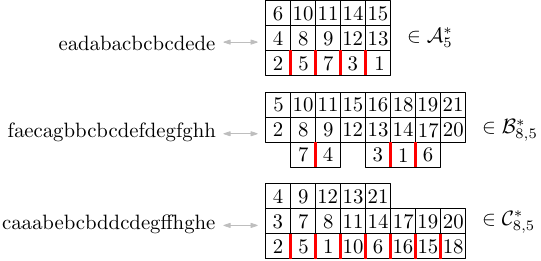}
    \caption{A word in $\mathcal{A}_5$ and its corresponding Young tableau in $\mathcal{A}^*_5$ under the bijection; a word in $\mathcal{B}_{8,5}$ and its corresponding Young tableau in $\mathcal{B}^*_{8,5}$; and a word in $\mathcal{C}_{8,5}$ and its corresponding Young tableau in $\mathcal{C}^*_{8,5}$.}
    \label{abcbijection}
\end{figure}

\end{document}